\documentclass[a4paper]{cas-sc}

\def\tsc#1{\csdef{#1}{\textsc{\lowercase{#1}}\xspace}}
\tsc{WGM}
\tsc{QE}
\tsc{EP}
\tsc{PMS}
\tsc{BEC}
\tsc{DE}

\usepackage[square,numbers]{natbib}      
\usepackage{comment}
\usepackage{subfig}
\usepackage{caption}
\usepackage{accents}
\usepackage[section]{placeins}

\usepackage[ruled,vlined,linesnumbered]{algorithm2e}

\usepackage[toc,page]{appendix}

\newcommand{\R}{\mathbb{R}}

\newcommand{\N}{\mathbb{N}}

\newcommand{\bfo}{\mathbf{1}}

\newcommand{\sk}{\mathbf{\Omega}}

\newcommand{\diag}{\operatorname{diag}}
\newcommand{\cH}{\mathbf{H}}

\newcommand{\casescale}[1]{%
  \begingroup
  \setlength{\arraycolsep}{4pt}
  \renewcommand{\arraystretch}{1.05}
  \scalebox{0.95}{$\displaystyle #1$}
  \endgroup
}

\usepackage{color}
\usepackage{todonotes}
\newtheorem{theorem}{Theorem}

\newtheorem{proposition}[theorem]{Proposition} 
\newtheorem{corollary}[theorem]{Corollary}
\newtheorem{remark}{Remark} 

\newproof{proof}{Proof}

\newtheorem{lemma}[theorem]{Lemma}
\newtheorem{example}{Example}

\begin{document}
\let\WriteBookmarks\relax
\def\floatpagepagefraction{1}
\def\textpagefraction{.001}
\shorttitle{On Zero-sum Game Representation for Replicator Dynamics}
\shortauthors{H. Yin et~al.}

\title [mode = title]{On Zero-sum Game Representation for Replicator Dynamics}            



\author[1]{Haoyu Yin}[orcid=0000-0002-0181-3922]

\ead{h.yin@wustl.edu}
\cormark[1]


\affiliation[1]{organization={Department of Electrical and Systems Engineering},
                addressline={Washington University in St. Louis}, 
                postcode={63130}, 
                state={MO},
                country={United States of America}}

\author[1]{Xudong Chen}[orcid=0000-0002-0135-0606]
\ead{cxudong@wustl.edu}

\author[1]{Bruno Sinopoli}[orcid=0000-0001-5778-4879]
\ead{bsinopoli@wustl.edu}

\cortext[cor1]{Corresponding author}


\begin{abstract}             
Replicator dynamics have been widely used in evolutionary game theory to model how strategy frequencies evolve over time in large populations. 
The so-called payoff matrix encodes the pairwise fitness that each strategy obtains when interacting with every other strategy, and it solely determines the replicator dynamics. 
If the payoff matrix is unknown, we show in this paper that it cannot be inferred from observed strategy frequencies alone --- distinct payoff matrices can induce the same replicator dynamics. 
We thus look for a canonical representative of the payoff matrix in the equivalence class. The main result of the paper is to show that for every polynomial replicator dynamics (i.e., the vector field is a polynomial), there always exists a skew-symmetric, polynomial payoff matrix that can induce the given dynamics. 
\end{abstract}



\begin{keywords}
Replicator Dynamics \sep Population Games \sep Zero-sum Games
\end{keywords}
\maketitle
\section{Introduction}
The so-called replicator dynamics, originally introduced in~\cite{hauert2002replicator}, 
{  have been widely used} in evolutionary game theory~\cite{weibull1997evolutionary,hofbauer1998evolutionary} to describe how strategy frequencies change over time in a large population. 
Consider the basic setting where every agent in the population has $n$ (common) pure strategies at its disposal at any time. Denote by $x_i(t)$ the proportion of the agents taking the $i$th strategy at time $t$. Let $x(t) := (x_1(t),\ldots,x_n(t))$ be the system state. Note that $\sum_{i = 1}^n x_i(t) = 1$ and $x_i(t) \geq 0$ for all $i = 1,\ldots, n$, so $x(t)$ is a probability vector. 
The agents in the population randomly pair with each other to play symmetric games whose payoff matrix is given by $H(x(t))\in \R^{n\times n}$ (i.e., the payoff matrix is a function of the current state $x(t)$). The payoff matrix encodes the competitive/cooperative nature between different strategies. Specifically, the $ij$th entry of $H(x(t))$ represents the payoff received by any agent using strategy $i$ when it interacts with an opponent using strategy $j$ with $x(t)$ the population state.

When playing the symmetric game, each agent evaluates the incremental benefit of abandoning its current strategy $i$ in favor of an alternative $k$. Conditional on facing an opponent who plays $j$, this benefit is the payoff difference $H_{k j}(x(t))-H_{i j}(x(t))$. Note that the strategy revision depends exclusively on such differences. Thus, it is invariant to any column-wise uniform shift of the payoff matrix (i.e., only the relative values of the entries of the $j$th column of $H(x(t))$ matter). 
Agents then follow a prescribed learning rule (e.g., pairwise proportional comparison and imitation driven by dissatisfaction~\cite{sandholm2010population}) to revise their strategies. 
Taking an appropriate limit (see, e.g.,~\cite{kara2023differential} for details), we arrive at the following ordinary differential equation for~$x(t)$:
\begin{equation}\label{eq:replicatordynamics}
\dot{x}_i(t)=x_i(t)(p_i(x(t))-\bar{p}(x(t))), \quad \mbox{for all } i = 1,\ldots, n, 
\end{equation}
where $p(x) = (p_1(x),\ldots, p_n(x)):= H(x)x$ is known as the {\em payoff function}~\cite{park2019payoff} and $\bar{p}(x):=\sum_{j=1}^nx_jp_j(x)$ is the {\em average payoff}. The dynamics of $x(t)$ evolve on the standard simplex $\Delta^{n-1} := \{x\in \R^n \mid x_i \geq 0 \mbox{ and } \sum_{i = 1}^n x_i = 1\}$.

In addition to population games, we note that replicator dynamics also arise in the context of online learning~\cite{falniowski2025discrete,mertikopoulos2018cycles,sorin2020replicator} and reinforcement learning~\cite{hennes2020neural,paul2024multi}. 
To elaborate, consider the basic setting of online learning where a single agent has $n$ actions at its disposal, and its goal is to learn a mixed strategy $x(t)\in \Delta^{n-1}$ in real-time to minimize the expected regret, with the action-utility vector $p(x(t))\in \R^n$ revealed at each time-step. 
For certain learning rule such as FoReL~\cite{mertikopoulos2018cycles}, it has been shown that the update equation of $x(t)$ can be approximated by the replicator dynamics as the continuous-time limit. 

It is clear that the payoff matrix $H(x)$  determines completely the replicator dynamics~\eqref{eq:replicatordynamics}. 
Many existing works in the literature assume that $H(x)$ has an explicit expression, and focus on the system behavior of the resulting replicator dynamics (e.g., characterization of  the attractors and the phase portraits). 

In this paper, we take a different perspective: Assuming that one can only observe the dynamics of $x(t)$, we investigate whether it is feasible to identify the unknown payoff matrix $H: \Delta^{n-1} \to \R^{n\times n}$? 
Said in other words, we ask for any given replicator dynamics, is there a unique payoff matrix associated with it? 
If the answer is negative, then is there a ``canonical'' representative (we will make the canonicity clear)? 
The problems we pose play important roles in the setting where a monitor in their study of ecosystem aims to uncover the competitive/cooperative relations between different species, or, the setting where a central planner in the congestion game aims to understand the agents' incentives against others' choices of strategies.
Specifically, we show that for a generalized class of replicator systems
$\dot x = \diag(x)\,g(x),
x\in\Delta^{n-1},$
where \(g:\R^{n}\to\R^{n}\) is a polynomial vector field satisfying
\(x^{\!\top}g(x)=0\) for all \(x\in\Delta^{n-1}\), there exists a skew-symmetric matrix $A(x)$ such that $g(x)=A(x)\,x,
\text{with }
A(x)=-A(x)^{\!\top}$. 
{  
In other words, we show that every population game has its own zero-sum game representation. 
In particular, the skew-symmetric payoff matrices capture pure competition in which every gain for one strategy is matched by an equal loss to another. } 
This structure yields a benchmark devoted to competitive dynamics while remaining tractable for the analysis of stability, coexistence~\cite{mai2018cycles}, and learning behavior~\cite{mertikopoulos2018cycles}. 
For constant payoffs, the canonical case is the three-strategy {  Rock-Paper-Scissors (RPS) games and their}
$n$-strategy generalizations. The associated replicator dynamics exhibit recurrent interior orbits on the simplex and constitute the standard framework in which population permanence~\cite{amann1985permanence,hofbauer1987permanence} has been rigorously established. 

A natural question is whether the properties for the case of constant skew-symmetric $H$ still hold when $H(x)$ is now state-dependent (but still skew-symmetric). Our analysis reveals that properties like permanence can fail in this broader setting. 
In fact, one can have a skew-symmetric payoff matrix $H(x)$, yet the associated replicator dynamics are gradient flow (see Example~\ref{example:1}). More surprisingly, we will see that the class of skew-symmetric payoff matrices is rich enough to generate all replicator dynamics. More specifically, we show that for every polynomial replicator dynamics, one can always associate to it a skew-symmetric matrix payoff function $H(x)$. The precise statement will be given in Theorem~\ref{thm:main} as the main result of the paper.


\paragraph{Literature review.} 
{The replicator dynamics has been studied in various scenarios. In~\cite{mai2018cycles,skoulakis2021evolutionary}, the authors investigated the co-evolution of the replicator system where the payoff matrix is generated from a dynamic system based on a generalized RPS game. The natural transformation~\cite{hofbauer1998evolutionary} was performed to study the cyclic behavior.  In~\cite{park2019payoff,park2019population}, the authors studied a more general setting where the payoff function is generated from a dynamic system, and the passivity theorem~\cite{pavel2022dissipativity,arcak2020dissipativity} is used to find the equilibria of the system. The most related work to ours is the static payoff model~\cite{park2018passivity}, where the payoff is a static function. For $p(x)$ being static,}
the dynamic behavior of~\eqref{eq:replicatordynamics} has been extensively investigated in the literature.  Many existing works have focused on the case where the payoff matrix $H(x)$ is constant: For $H$ a symmetric matrix, it has been shown~\cite{sandholm2010population} that the replicator dynamics~\eqref{eq:replicatordynamics} are gradient flows of the potential function~$\bar{p}(x)$ with respect to the Shahshahani metric. 
For $H$ skew-symmetric, it has been shown~\cite{biggar2024attractor} that there is a unique global attractor of~\eqref{eq:replicatordynamics} (roughly speaking, an attractor is a minimal compact set that is forward invariant and asymptotically stable).  
For the case where $H$ is a $3$-by-$3$ skew-symmetric matrix, a complete characterization of the attractors has been obtained by Zeeman~\cite{zeeman2006population}. 
In fact, more has been shown in~\cite{zeeman2006population} as well as in the follow-up work~\cite{bomze1983lotka,bomze1995lotka} for $H$ a general $3$-by-$3$ matrix. Amongst others, they have presented, as a conjecture with partial proofs, all possible phase portraits (49 in total) for the resulting replicator dynamics with isolated equilibria and periodic orbits.

The literature is relatively sparse for the case where the payoff matrix $H(x)$ depends on the state~$x$, which is more realistic for a number of settings. For example, such dependence arises naturally in modeling the ecosystem where the interactions among different species are often in higher order~\cite{battiston2020networks} (in this context, the $x_i(t)$'s represent the normalized population densities of the species). 
Also, for congestion games~\cite{sandholm2010population} where different strategies represent different routes between the destinations, it is easy to see that the payoff matrix $H(x)$ depends on $x$ (the $x_i$'s represent the strategy frequencies). 
We further mention the demand-response program for smart grids~\cite[Example~2]{park2019payoff} where the payoff matrix $H(x)$ depends on the user-strategy frequencies. 
The analysis of replicator dynamics for non-constant $H(x)$ is, however, much more involved and can quickly lose tractability. Results have been obtained for some special $3$-by-$3$ payoff matrices $H(x)$ in~\cite{wesson2015replicator} where $H(x)$ is an affine function. Specifically, the authors have exhibited all the equilibria of the resulting replicator dynamics, analyzed their local stability by evaluating the Jacobian matrices, and carried out numerical studies for the phase portrait.    

The remainder of the paper is organized as follows: The notation used throughout the analysis is collected at the end of this section. In section~\ref{sec:problem_formulation}, we introduce a motivating example for our problem and state the main result followed by a sketch of its proof.  
We present the proof of our main result in section~\ref{sec:main_result}, where the theoretical developments are illustrated in three subsections. The paper ends with a summary and outlook in section~\ref{sec:conclusion}.

\paragraph{Notation.}
We introduce the notation that will be used throughout the analysis. We define the hyperplane 
$$\cH:=\{x\in \mathbb{R}^n|\sum_{i=1}^n x_i=1\}.$$ 
Given arbitrary integer $d> 0$, we let $\sk_{[d]}^{k\times k}$ be the vector space of all $k\times k$ skew-symmetric matrices $A$ with $A^\top =-A$, whose entries are polynomials in $x \in \mathbb{R}^n$ of degree at most $d$, and $\boldsymbol{\Omega}_{(d)}^{k \times k}$ denote the homogeneous subspace of degree $d$. If $d=0$, we write $\sk^{n \times n}$ for simplicity.

Let $\mathbf{P}_{[d]}^{k\times \ell}$ be the space of $k \times \ell$ matrices whose entries are polynomials in variable $x\in \R^n$ of degree at most $d$, and by $\mathbf{P}_{(d)}^{k\times \ell}$ the space of $k\times \ell$ matrices whose entries are homogeneous polynomials in $x\in \R^n$ of degree $d$. 
For the case where $\ell = 1$, we simply write $\mathbf{P}_{[d]}^{k}$ and $\mathbf{P}_{(d)}^{k}$. It should be clear that 
{ 
$$ { \sk^{k\times k}_{[d]} = \oplus_{m = 0}^d \sk^{k\times k}_{(m)} \ \mbox{  and  }\ } 
\mathbf{P}^{k\times \ell}_{[d]} = \oplus_{m = 0}^d \mathbf{P}^{k\times \ell}_{(m)}. 
$$
For $d<0$, we let the corresponding vector space be the trivial space, i.e., the space with only the $0$ vector.
}

{  Denote by $\N_0$ the set of nonnegative integers.}     
We use multi-index $x^\alpha$ to denote the monomial $x^\alpha:=x_1^{\alpha_1} x_2^{\alpha_2} \cdots x_n^{\alpha_n}$, where $ \alpha=\left(\alpha_1, \ldots, \alpha_n\right) \in \mathbb{N}^n_0$ 
and $|\alpha|:=\sum_{i=1}^n \alpha_i$ is the
degree of the monomial. 
We say that a scalar polynomial contains a monomial if the coefficient of the monomial in the standard multivariate expansion of the polynomial is non-zero. 
The degree of the polynomial is by convention the largest degree of a monomial contained in it. 
For a vector- or matrix-valued polynomial, we define its degree to be the largest degree among all of its entries.

\section{Motivating Example and Main Result} 
\label{sec:problem_formulation}
\subsection{Motivating Example}
\label{subsec:motivating_example}

To proceed, we first formulate the identifiability problem precisely. 
Consider the following map: 
{ \begin{equation}
\label{eq:def_phi}
\phi: H(x) \mapsto f(x):= \diag(x) (H(x) x - x^\top H(x) x\mathbf{1}), 
\end{equation}}
whose input is a payoff matrix $H(x)$ and whose output is the vector field $f(x)$ of the associated replicator dynamics~\eqref{eq:replicatordynamics}. 
Then, the question we posed earlier about the uniqueness of $H(x)$ can be translated to the following one: {\em Is $\phi$ injective?}

Of course, the answer depends on the domain of $\phi$. If we restrict the domain to the class of all constant payoff matrices, then the map $\phi$ is essentially injective. We can show that for any two constant matrices $H$ and $H'$, we have that $\phi(H) = \phi(H')$ if and only if $H = H' + \mathbf{1} v^\top $ where $\mathbf{1}\in \R^n$ is the vector of all ones and $v\in \R^n$ is an arbitrary vector. A proof of this fact can be found in Appendix~\ref{sec:appenA}. However, if we let the domain of $\phi$ be the space of all polynomial matrices, then the answer changes completely. 
To illustrate, we provide an example below.


\begin{example}
\label{example:1}
\normalfont
Given an arbitrary constant payoff matrix $H$, 
let $S:=\frac{H+H^{\top}}{2}$ and $\Omega:=\frac{H-H^{\top}}{2}$ be its symmetric part and skew-symmetric part, respectively. We then define an affine matrix-valued function $H'(x)$ as follows: 
\begin{equation}
\label{eq:construct_H_prime}
    H'(x)= S x \mathbf{1}^{\top}-\mathbf{1} x^{\top} S+ \Omega, \quad \mbox{for all } x\in \Delta^{n-1}.
\end{equation}
Note in particular that $H'(x)$ is skew-symmetric. 
It is not hard to see that $\phi(H'(x)) = \phi(H(x))$; indeed, setting $S:=\frac{H+H^{\top}}{2}$ and $\Omega:=\frac{H-H^{\top}}{2}$, we have that
\begin{align*}
\phi(H'(x)) & = \diag(x)\left( S x \mathbf{1}^{\top}x-\mathbf{1} x^{\top}Sx+\Omega x\right) \\
& =  \diag(x)\left((S+\Omega)x-\mathbf{1} x^{\top}Sx\right)
\\
&=\diag(x)\left((S+\Omega)x-\mathbf{1} x^{\top}(S+\Omega)x\right)\\
&=\phi(H),
\end{align*}
where the first equality follows from the definition of~$\phi$ in~\eqref{eq:def_phi} and that  $x^\top H'(x)x=0$, the second equality follows from the fact that $x^\top \bfo = 1$ for all $x\in \Delta^{n-1}$, the third equality follows from the fact that $x^\top \Omega x = 0$, and the last equality follows again from the definition of~$\phi$. 
Note that if $H$ is itself skew-symmetric, then $H'(x) \equiv H$. Otherwise, $H'(x)$ is essentially different from $H$ in the sense that $H'(x)$ cannot be obtained from $H$ by adding $\mathbf{1}v(x)^\top$, for some vector-valued function $v(x)$. 

For further illustration, we exhibit below four constant $3$-by-$3$ matrices $H$ and use the formula~\eqref{eq:construct_H_prime} to obtain the associated $H'(x)$: 
\vspace{.1cm}


\noindent{\it Case 1:}
\[
\casescale{
H =
\begin{bmatrix}
0 & 1 & 1\\
1 & 0 & 1\\
1 & 1 & 0
\end{bmatrix},\quad
H'(x)=
\begin{bmatrix}
0& -x_1 + x_2& -x_1 + x_3\\
x_1 - x_2& 0& -x_2 + x_3\\
x_1 - x_3& x_2 - x_3& 0
\end{bmatrix}
}
\]
{\it Case 2:}
\[
\casescale{
H= \begin{bmatrix}
1 & 1 & 0\\
1 & 0 & 1\\
0 & 1 & 1
\end{bmatrix},\quad
H'(x) =
\begin{bmatrix}
0& x_2 - x_3& x_1 - x_3\\
-x_2 + x_3 & 0 & x_1 - x_2\\
-x_1 + x_3 & -x_1 + x_2 & 0
\end{bmatrix}
}
\]
{\it Case 3:}
\[
\casescale{
H=\begin{bmatrix}
0 & 2 & 0\\
2 & 0 & 2\\
2 & 2 & 4
\end{bmatrix},\quad
H'(x)=\begin{bmatrix}
0& -2x_1 + 2x_2 - x_3& -x_1 - 3x_3 - 1\\
2x_1 - 2x_2 + x_3& 0& x_1 - 2x_2 - 2x_3\\
x_1 + 3x_3 + 1& -x_1 + 2x_2 + 2x_3& 0
\end{bmatrix}
}
\]
{\it Case 4:}
\[
\casescale{
H=\begin{bmatrix}
2 & -2 & 0\\
0 & 2 & 2\\
2 & 0 & -2
\end{bmatrix},\quad
H'(x)=\begin{bmatrix}
0& 3x_1 - 3x_2 - 1& - 3x_2 + 2x_3 \\
-3x_1 + 3x_2 + 1& 0& -x_1 + 2x_2 + 4x_3\\
 3x_2 - 2x_3 & x_1 - 2x_2 - 4x_3 & 0
\end{bmatrix}
}
\]

The phase diagrams of the replicator dynamics for the above four cases are given in Figure~\ref{fig:example}.
\hfill{\qed}

\end{example}

\begin{figure*}
    \centering
\includegraphics[width=1\linewidth]{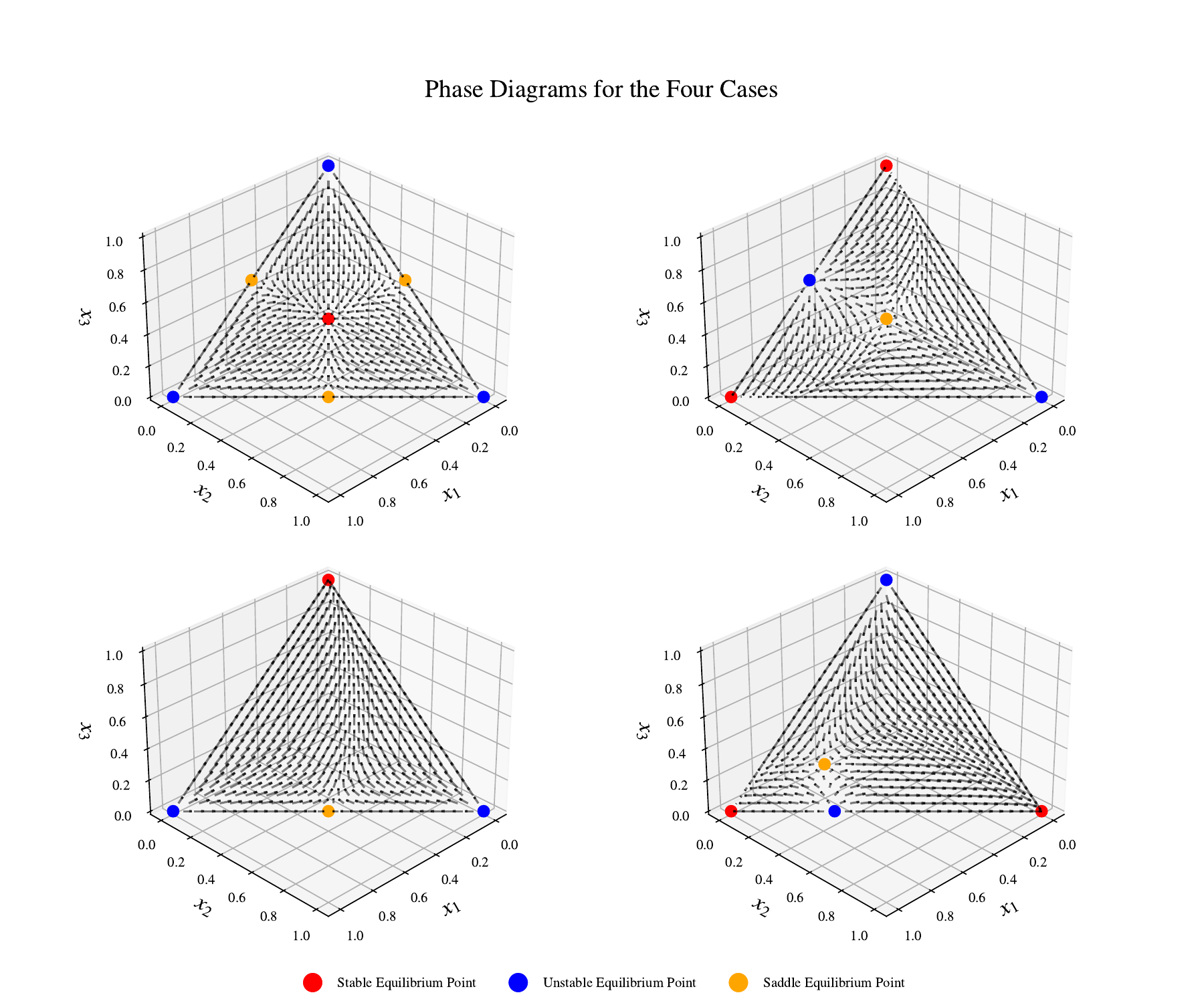}
    \caption{The Phase diagrams of the replicator dynamics induced by the payoff matrices in the four cases of Example~\ref{example:1}. 
    The two on the top are for cases $1$ and $2$ ---  these are gradient flows on the simplex $\Delta^{2}$ with the potential function $x^\top H x$ with respect to the Shahshahani metric. The two on the bottom are the phase diagrams for case $3$ and $4$.}
    \label{fig:example}
\end{figure*}

\subsection{Main Result}
\label{subsec:main}

We state below the main result of the paper:

\begin{theorem}\label{thm:main}
 Let $g: \R^n\to \R^n$ be a polynomial map of degree~$(d+1)$, with $d \geq 0$ arbitrary. 
 Suppose that $g(x)$ satisfies 
\begin{equation}\label{eq:hypothesis}
 x^\top g(x) = 0, \quad \mbox{for all } x\in \Delta^{n-1};
 \end{equation}
 then, there exists a polynomial map $A: \R^n \to \sk^{n\times n}$ of degree at most $d$ such that $g(x) = A(x) x$, { for all $x\in \mathbf{H}$}. 
\end{theorem}

\begin{remark}\normalfont Note that for any polynomial $g(x)$ that satisfies~\eqref{eq:hypothesis}, we have that
$$
x^\top g(x) = 0, \quad \mbox{for all } x\in \mathbf{H}.
$$
To wit, $x^\top g(x)$ is a polynomial function (and hence, analytic) and it is identically zero on $\Delta^{n-1}$, whose interior is open in $\mathbf{H}$. It then follows from the identity theorem~\cite{scheidemann2005introduction} that $x^\top g(x) = 0$ for all $x\in \mathbf{H}$. 
\end{remark}





We provide below a sketch of proof for Theorem~\ref{thm:main}. A complete proof will be presented in Section~\ref{sec:main_result}.
\vspace{.2cm}

\noindent
{\it Sketch of the proof:}  
We take a two-step approach to establish the result: In the first step, we consider a special class of polynomial maps $g(x)$ such that 
\begin{equation}\label{eq:specialcase}
x^\top g(x) = 0, \quad \mbox{for all } x \in \R^n.
\end{equation}
It is clear that all such $g(x)$ of degree $(d + 1)$ form a vector space, which we denote by $\mathbf{Q}_{(d+1)}^n$. 
{ 
We will provide a complete characterization of the vector space by exhibiting a spanning set of it. In particular, we will show in Lemma~\ref{lem:basis} that each element takes the following form}:
$$
v_{ij}^\alpha(x) := x^{\alpha + e_j} e_i - x^{\alpha+e_i} e_j,
$$
where $\alpha$ is some multi-index with $|\alpha| = d$. It then follows that for every such element $v_{ij}^\alpha(x)$, there exists a skew-symmetric polynomial matrix $$A^\alpha_{ij}(x) := x^\alpha (e_ie_j^\top - e_j e_i^\top)$$ such that $v_{ij}^\alpha(x) = A^\alpha_{ij}(x) x$. 
Note that the two nonzero entries of $A_{ij}^\alpha$ are polynomials of degree $d$. 
Consequently, for any element $g\in \mathbf{Q}^n_{(d+1)}$, we conclude that
$$g(x) = \sum_{i,j,\alpha} c_{ij}^\alpha v_{ij}^\alpha(x) = \left [ \sum_{i,j,\alpha} c_{ij}^\alpha A_{ij}^\alpha(x)\right ] x =: A(x)x,$$
where $c_{ij}^\alpha \in \R$ are some constants and $A(x):= \sum_{i,j,\alpha} c_{ij}^\alpha A_{ij}^\alpha(x)$ is skew-symmetric. The above analysis will be carried out in Subsection~\ref{ssec:specialcase}.   

Then, in the second step, we introduce a method, termed the reduction technique, that can reduce the general case (i.e., the case where $g$ satisfies~\eqref{eq:hypothesis} but not necessarily~\eqref{eq:specialcase}) to the special case. Specifically, we demonstrate in Subsection~\ref{ssec:reductiontechnique} that for any given $g$ of degree $(d+1)$, there exists another polynomial map $h:\R^n \to \R^n$ of degree at most $(d+1)$ such that 
\begin{equation}\label{eq:g_to_h}
x^\top g(x) = x^\top h(x), \quad \mbox{for all } x\in \R^n,
\end{equation}
and, moreover, there exists a skew-symmetric polynomial matrix $B(x)$ of degree at most $d$ such that 
\begin{equation}\label{eq:h_to_B}
h(x) = B(x)x,   {\text{ for all }x\in \mathbf{H}.}
\end{equation}
These two facts~\eqref{eq:g_to_h} and~\eqref{eq:h_to_B} will be established in Lemma~\ref{lem:s_d} and Lemma~\ref{lem:H_Bx}, respectively. 

Theorem~\ref{thm:main} is now an immediate consequence of the above arguments; indeed, given $g$, we write
$$
g = h + g', \quad \mbox{where } g':= g-h. 
$$
Note that $g'$ satisfies $x^\top g'(x) = x^\top (g(x) - h(x)) = 0$. Thus, from step~1, there exists a skew-symmetric polynomial matrix $A(x)$ of degree at most $d$ such that $g'(x) = A'(x)x$. From step~2, we have that $h(x) = B(x)x$. We conclude that 
$$g(x) = B(x) x + A'(x)x =: A(x)x,$$
where $A(x):= B(x) + A'(x)$ is a desired skew-symmetric polynomial matrix.

{ {






}}

\section{Proof of the Main Result}
\label{sec:main_result}

This section is dedicated to the proof of Theorem~\ref{thm:main}. As mentioned earlier, we take a two-step approach to establish the result and divide the section into three subsections: In  Subsection~\ref{ssec:specialcase},  we focus on the special case where $x^\top g(x) = 0$ for all $x \in \R^n$ (in general, it only holds that $x^\top g(x) = 0$ for $x\in \mathbf{H}$). Next, in Subsection~\ref{ssec:reductiontechnique}, we develop the technique that can reduce the general case to the special case. The arguments will then be combined in Subsection~\ref{ssec:sum_proof} to establish Theorem~\ref{thm:main}.

\subsection{Proof for the Special Case: $x^\top g(x)=0$ for all $x\in \mathbb{R}^n$}
\label{ssec:specialcase}

In this subsection, we focus on a special class of $g$ such that $x^\top g(x)=0$ for all $x\in \mathbb{R}^n$. For convenience, we introduce the following vector spaces for all $d\geq 0$:
\begin{equation}
\label{eq:def_G}
    \mathbf{Q}_{(d)}^n:=\left\{g(x)\in \mathbf{P}_{(d)}^{n} \mid x^{\top} g(x)=0, \mbox{ for all } x \in \R^n\right\},
\end{equation}
and let
{ 
$$
 \mathbf{Q}_{[d]}^n := \oplus_{m = 0}^d \mathbf{Q}_{(m)}^n.
$$}
We establish the following proposition: 

\begin{proposition}
\label{prop:g_skew}
    For any $d \geq 0$ and any $g\in \mathbf{Q}_{[d+1]}^n$, there exists an $A(x)\in \sk^{n\times n}_{[d]}$ such that $g(x) = A(x)x$. 
\end{proposition}

To establish the proposition, we need a few preliminary results. 
Given $g\in \mathbf{Q}_{(d)}^n$, we define 
\begin{equation*}
    \mathcal{I}_0(g):= \{i\in \{1,2,\ldots,n\} \mid g_i = 0 \} \quad  \mbox{and} \quad \mathcal{I}_1(g):=\{1,2,\ldots,n\} - \mathcal{I}_0(g).
\end{equation*}
In words, $\mathcal{I}_0(g)$ is the collection of indices~$i$ such that $g_i(x)$ is identically zero and $\mathcal{I}_1(g)$ is its complement. 
We first have the following lemma: 


\begin{lemma}
\label{lem:no_cancel}
    Let $g \in \mathbf{Q}^n_{(d+1)}$ be nonzero. Then, for any $i\in \mathcal{I}_1(g)$, $g_i(x)$ does not contain any monomial $x^\alpha$, with $\alpha_j = 0$ for all $j\in \mathcal{I}_1(g) - \{i\}$.

\end{lemma}

\begin{proof}
 The proof is carried out by contradiction. Suppose that $g_i(x)$ contains such a monomial $x^\alpha$, with $\alpha_j = 0$ for all $j \in \mathcal{I}_1(g) - \{i\}$; then, by the fact that $x^\top g(x) = 0$, there must exist an index $i'\in \mathcal{I}_1(g) - \{i\}$, such that $g_{i'}(x)$ contains a monomial $x^{\alpha'}$ with 
    \begin{equation}\label{eq:monomial_idx}
    \alpha + e_i = \alpha' + e_{i'}.
    \end{equation}
    Note that $\alpha_{i'} = 0$. It follows that the $i'$th entry of the left-hand side of~\eqref{eq:monomial_idx} is $0$ while the $i'$th entry of the right-hand side is at least $1$, which is a contradiction.  
\hfill{\qed}\end{proof}

With the lemma above, we characterize the vector space $\mathbf{Q}^n_{(d+1)}$ in the following lemma: 

\begin{lemma}
\label{lem:basis}
The space $\mathbf{Q}^n_{(d+1)}$ is spanned by the following vectors: 
    \begin{equation}
    \label{eq:def_basis}
    \mathcal{V}:=\left\{ v_{ij}^\alpha(x) =x^{\alpha+{e}_{j}} e_i-x^{\alpha+{e}_{i}} {e}_{j} \mid |\alpha|=d,  \right.  \left. 1 \leq i < j \leq n\right\}.
    \end{equation}
\end{lemma}

\begin{proof}
First, note that any element $v^\alpha_{ij}(x)$ belongs to $\mathbf{Q}^n_{(d+1)}$; indeed, 
\begin{equation}
\label{eq:basis_in_space}
    x^{\top} v_{ij}^\alpha(x)=x^\alpha\left(x_ix_j-x_jx_i\right)=0.
\end{equation}
We now show that any $g \in \mathbf{Q}^n_{(d+1)}$ can be written as a linear combination of the $v^\alpha_{ij}$'s. 

To proceed, we consider the polynomial $g_1(x)$. By Lemma~\ref{lem:no_cancel}, $g_1(x)$ does not contain $x_1^{d+1}$. Thus, we can write
\begin{equation}
\label{eq:g1expression}
g_1(x) = \sum_{|\alpha|= d}\sum_{j>1} c_{\alpha, j} x^{\alpha + e_j},
\end{equation}
where $c_{\alpha, j} \in \R$ are constants. Next, we define
$$g'(x):= g(x) - \sum_{|\alpha|=d} \sum_{j> 1}c_{\alpha, j} {v}_{1j}^\alpha(x)$$
By construction, $g'_1$ is identically zero and, moreover, satisfies
$$
    x^\top g'(x) = x^\top g(x) - \sum_{j> 1}\sum_{\substack{|\alpha|=d }} c_{\alpha, j} x^\top v_{1j}^\alpha(x) = 0, \forall x\in \R^n. 
$$

If $g' = 0$, then the proof is done. We thus assume that there is at least one nonzero entry, say $g'_2$. 
Using again Lemma~\ref{lem:no_cancel}, $g'_2$ does not contain any monomial $x_2^\alpha$ for $\alpha_j = 0$ for $j > 2$. Similarly, we can write
{ $$
g'_2(x) = \sum_{|\alpha| = d} \sum_{j > 2} c'_{\alpha, j} x^{\alpha + e_j},
$$}
and define
$$
g''(x) := g'(x) - \sum_{\substack{|\alpha|=d }} \sum_{j> 2}c'_{\alpha, j} {v}_{2j}^\alpha(x).
$$
Since the first entry of $v^\alpha_{2j}$, for $j > 2$, is $0$ and since $g'_1 = 0$, we have that $g''_1 = 0$. By construction, $g''_2$ is zero as well. Moreover, $x^\top g''(x) = 0$.

Iterating the above procedure, we obtain an element $g^*\in \mathbf{Q}^n_{(d+1)}$ such that $g^*_i(x)= 0$ for $i = 1,\ldots, n-1$ and, moreover, the difference $(g - g^*)$ is a linear combination of the elements $v^\alpha_{ij}$.  
But then, 
$x_ng^*_n(x) = x^\top g^*(x) = 0$ for all $x\in \R^n$. Thus, it must hold that $g^*_n = 0$ and hence, $g^* = 0$. This completes the proof.  
\hfill{\qed}\end{proof}

{ 

The following result computes the dimension of $\mathbf{Q}^n_{(d+1)}$.






\begin{corollary}
\label{collary:basis}
    The dimension of $\mathbf{Q}^n_{(d+1)}$ is given by
    $$\dim \mathbf{Q}_{(d+1)}^n = n\binom{n+d}{d+1}-\binom{n+d+1}{d+2}.$$
\end{corollary}

\begin{proof}
    Consider the linear map $L:\mathbf{P}_{(d+1)}^n\rightarrow \mathbf{P}_{(d+2)}^1$ given by $L(g(x)):=x^\top g(x)$. It follows that  $\mathbf{Q}_{(d+1)}^n$ is the kernel of $L$. 
    We show that $L$ is surjective. 
   For any given $q(x)\in \mathbf{P}_{(d+2)}^1$, let
$g(x):=\frac{1}{d+2}\frac{\partial q(x)}{\partial x}\in \mathbf{P}_{(d+1)}^n$. 
Since $q(x)$ is a homogeneous polynomial of degree $(d+2)$, it follows from Euler's identity~\cite{euler2012elements} that 
$L(g(x))=q(x)$. 
Then, by the rank-nullity theorem, we have that
\begin{equation}\label{eq:dimQ}
\dim \mathbf{Q}_{(d+1)}^n
    =\dim \mathbf{P}_{(d+1)}^n-\dim \mathbf{P}_{(d+2)}^1.
\end{equation} 
To evaluate $\dim \mathbf{P}_{(d+1)}^n$, 
we note that the number of monomials $x^\alpha$, with $|\alpha|=d+1$, is $\binom{n+d}{d+1}$ (known as the stars-and-bars problem~\cite{flajolet2009analytic}), so 
\begin{equation}\label{eq:dim_1}
\dim \mathbf{P}_{(d+1)}^n = n\binom{n+d}{d+1}.
\end{equation}
Similarly, 
\begin{equation}
\label{eq:dim_2}
\dim \mathbf{P}_{(d+2)}^1=\binom{n+d+1}{d+2}.
\end{equation}
The corollary then follows from~\eqref{eq:dimQ},~\eqref{eq:dim_1} and~\eqref{eq:dim_2}. \hfill{\qed}
\end{proof}
}

Now, for a given multi-index $\alpha$ with $|\alpha| = d$ and a given pair of indices $(i,j)$, with $1 \leq i < j \leq n$, we define $A^\alpha_{ij}\in \mathbf{\Omega}^{n\times n}_{(d)}$  as follows
$$
A_{i j}^\alpha(x):=x^\alpha\left(e_i e_j^{\top}-e_j e_i^{\top}\right).
$$
We have the following lemma, which follows directly from computation: 

\begin{lemma}
\label{lem:homo_g}
    For any element $v^\alpha_{ij}\in \mathcal{V}$, we have that
    $$
    v^\alpha_{ij}(x)= A^\alpha_{ij}(x)x. 
    $$
\end{lemma}

With the preliminary results above, we will now establish Proposition~\ref{prop:g_skew}.

\begin{proof}[Proof of Proposition~\ref{prop:g_skew}]
    { Given $g\in \mathbf{Q}^n_{[d+1]}$,} we write
    $$
        g(x) = g_{(d+1)} + g_{(d)} + \cdots + g_{(0)},
    $$
    where the entries of $g_{(m)}\in \mathbf{P}^n_{(m)}$ contain only monomials of degree~$m$. It suffices to show that for each $m = 0,\ldots, d+1$, there exists a skew-symmetric polynomial matrix $A_{m}(x)$ of degree {  at most~$(m-1)$} such that $g_{(m)}(x) = A_m(x) x$. Consider two cases: 
    \vspace{.1cm}

    \noindent
    {\it Case 1: $m = 0$.} Note that $g_{(0)}\in \R^n$ is nothing but a vector. Since $x^\top g_{(0)} = 0$ for all $x\in \R^n$, we must have that $g_{(0)} = 0$. We then set $A_0 = 0$. 
    \vspace{.1cm}

    \noindent
    {\it Case 2: $1\leq m \leq d+1$.} By Lemma~\ref{lem:basis}, we can write $g_{(m)}$ as a linear combination of the $v^\alpha_{ij}$'s, for $|\alpha| = m - 1$ and $1\leq i < j \leq n$. By Lemma~\ref{lem:homo_g}, every such $v^\alpha_{ij}$ obeys $v^\alpha_{ij}(x) = A^\alpha_{ij}(x)x$, where the degree of $A^\alpha_{ij}$ is at most $(m-1)$. It then follows $g_{(m)} = A_m(x)x$, where $A_m$ is obtained by taking the linear combinations of the $A^\alpha_{ij}$, with the same set of coefficients. The degree of $A_m$ is at most $(m-1)$. 
    \hfill{\qed}\end{proof}

With all the results above, we present Algorithm~\ref{alg:Aprime} to illustrate the procedure of deriving the skew-symmetric matrix $A'(x)$ from $g(x)\in\mathbf{Q}_{[d+1]}^n$.
{  In the for-loop (line $3-10$), computing the coefficients reduces to solving linear systems whose number of equalities matches the number of basis vectors (via elimination). By Corollary~\ref{collary:basis}, the dimension of $\mathbf{Q}_{[d+1]}^n$ is given by 

{ 
$$
N:= \dim \mathbf{Q}_{[d+1]}^n = \sum_{k = 0}^{d} \dim \mathbf{Q}_{(k+1)}^n = n\sum_{k=0}^d \binom{n+k}{k+1}-\sum_{k=0}^d\binom{n+k+1}{k+2}=(n-1)\binom{n+d+1}{d+1}-\binom{n+d+1}{d+2}+1.
$$ }
Using Gaussian elimination, solving such systems requires $O(N^3)$ arithmetic operations~\cite{golub2013matrix}. Hence, Algorithm~\ref{alg:Aprime} has overall complexity of $O(N^3)$, which is polynomial in $n$ and exponential in $d$. }

\begin{algorithm}
\DontPrintSemicolon      
\KwIn{ $g(x) \in \mathbf{Q}_{[d+1]}^n$}

  Decompose $g(x)=\sum_{m=1}^{d+1}g_{(m)}(x)$\;

  \For{$m\gets 1$ \KwTo $d+1$}{  
  construct multi-index $\alpha$, { spanning set of $v_{ij}^{\alpha}(x)$, with $|\alpha|=m-1$}\;
  set $c_{\alpha,j}^{(i)}\gets 0$\;
\For{$i\gets 1$ \KwTo $n-1$}{          
    compute $c_{\alpha,j}$ such that 
    $     
        e_i^{\top}\Bigl(g_{(m)}(x)-\sum_{|\alpha|=m-1}\sum_{j>i} c_{\alpha,j}\,v_{ij}^{\alpha}(x)\Bigr)=0$\;

    set $c_{\alpha,j}^{(i)}\gets c_{\alpha,j}$\;
    
    set $g_{(m)}(x)\gets g_{(m)}(x)-\displaystyle\sum_{|\alpha|=m-1}\sum_{j>i} c_{\alpha,j}^{(i)}\,v_{ij}^{\alpha}(x)$\;
}
construct $A_{ij}^{\alpha}(x)$ from $v_{ij}^{\alpha}(x)$ via Lemma~\ref{lem:homo_g}\;
set $A'_{(m)}(x)\gets\displaystyle\sum_{i,j}\sum_{|\alpha|=m-1}c_{\alpha,j}^{(i)}\,A_{ij}^{\alpha}(x)$ \;
}
Compute $A'(x)=\sum_{m=1}^{d+1}A'_{(m)}$\;
\KwOut{Skew-symmetric payoff matrix $A'(x)$}
\caption{Pseudo-code for deriving $A'(x)$ from $g(x)$}
\label{alg:Aprime}
\end{algorithm}

\subsection{The Reduction Technique}
\label{ssec:reductiontechnique}
In this section, we establish the following proposition: 

\begin{proposition}
\label{prop:reduction}
Let $g\in \mathbf{P}^n_{[d+1]}$ be such that there is no constant term in $g$ and that $x^\top g(x) = 0$ for all $x\in \mathbf{H}$. Then, there exists an $h\in \mathbf{P}^n_{[d+1]}$ such that the following two items are satisfied: 
\begin{enumerate}
\item For any $x\in \R^n$, we have that $x^\top g(x) =  x^\top h(x)$;
\item There exists a $B\in \mathbf{\Omega}^{n\times n}_{[d]}$ such that $h(x) = B(x)x$, {  for all $x\in \mathbf{H}$}. 
\end{enumerate}
\end{proposition}

We need a sequence of lemmas to establish the result. We start with the following:

\begin{lemma}
\label{lem:factor}
Under the hypothesis of Proposition~\ref{prop:reduction}, there exists a scalar polynomial $s(x)$ of degree at most~$(d+1)$ such that
\begin{equation}\label{eq:g_to_s}
x^\top g(x)=\left(1-x^\top \mathbf{1}\right) s(x), \quad \mbox{for all }x\in \mathbb{R}^n.
\end{equation}
Furthermore, $s(x)$ does not contain a constant nor a linear term.
\end{lemma}

\begin{proof}
    First, we establish~\eqref{eq:g_to_s}. Consider a linear change of variable $y = T x$, for $T\in \R^{n\times n}$ nonsingular and $y_1 := x^\top \mathbf{1}$ (i.e., the first row of $T$ is $\mathbf{1}^\top$). Then, $x^\top g(x)$ can be written as a polynomial in $y$, which we denote by { $q(y)$}. Since $x^\top g(x) = 0$ for all $x\in \mathbf{H}$, we have that { $q(y)=0$} for all $y$ with $y_1 = 1$. It follows that $(1 - y_1)$ is a factor of { $q(y)$} and hence, $(1 - x^\top \mathbf{1})$ is a factor of $x^\top g(x)$.

    Next, we show that $s$ does not contain constant nor linear term. This is done by degree matching for the two sides of~\eqref{eq:g_to_s}. On one hand, by hypothesis, $g$ does not contain a constant term, which implies that all monomials in $x^\top g(x)$ are of degree at least~$2$. On the other hand, if $s$ has a nonzero constant term $s_{(0)}$, then $s_{(0)}$ will also be the constant term of $(1 - x^\top \mathbf{1}) s(x)$, which is a contradiction. Similarly, setting $s_{(0)} = 0$, we conclude that the linear term $s_{(1)}$ of $s$ is also the linear term of $(1 - x^\top \mathbf{1})s(x)$, and hence must be zero.     
\hfill{\qed}\end{proof}

We next have the following lemma:

\begin{lemma}
\label{lem:scalar_vec}
    Let $s(x)$ be given as in the statement of Lemma~\ref{lem:factor}. Then, there exists a symmetric $H\in \mathbf{P}_{[d-1]}^{n\times n}$ such that 
    $$s(x)=x^\top H(x) x.$$
\end{lemma}

\begin{proof}
    Since {  $s(x)$} does not contain a constant or a linear term,  we can write 
    $$
        s(x) = \sum_{2\leq |\alpha| \leq d+1} c_\alpha x^\alpha,
    $$
    where $c_\alpha\in \R$ are constants. 
    For any such $\alpha$ in the summation, we consider two cases: 
    \vspace{.1cm}

    \noindent
    {\it Case 1: There exists an $\alpha_i$ such that $\alpha_i\geq 2$.} In this case, we define
    $$H_\alpha(x):= x^{\alpha - 2e_i} e_ie_i^\top.$$

    \vspace{.1cm}

    \noindent
    {\it Case 2: $\alpha_i \leq 1$ for all $i \in \{1,2,\ldots,n\}$.} Since $|\alpha| \geq 2$, there exist two distinct indices $i,j\in \{1,2,\ldots,n\}$ such that $\alpha_i = \alpha_j = 1$. We then define
    $$
    H_\alpha(x) := \frac{1}{2} x^{\alpha - e_i - e_j} \left ( e_i e_j^\top + e_j e_i^\top \right ).
    $$

    Note that in either case, $H_\alpha(x)\in \mathbf{P}^{n\times n}_{[d-1]}$ is symmetric and, moreover, $x^\alpha = x^\top H_\alpha(x) x$ for all $x\in \R^n$.
    The proof is done by setting $H(x):= \sum_{2\leq |\alpha| \leq d + 1} c_\alpha H_\alpha(x)$. 
\hfill{\qed}\end{proof}

Now, with the two lemmas above, we define the element~$h\in \mathbf{P}^n_{[d+1]}$ as follows:
\begin{equation}\label{eq:h_def}
    h(x) := H(x) x - x^\top H(x) x \mathbf{1},
\end{equation}
where $H\in \mathbf{P}^{n\times n}_{[d-1]}$ is given in the statement of Lemma~\ref{lem:scalar_vec}. We show below that $h$ satisfies the two items of Proposition~\ref{prop:reduction}. This is done in Lemma~\ref{lem:s_d} and Lemma~\ref{lem:H_Bx}.

\begin{lemma}
\label{lem:s_d}
     Let $h$  be given as in~\eqref{eq:h_def}. Then, 
     $$x^\top g(x) = x^\top h(x), \quad \mbox{for all } x\in \R^n.$$
\end{lemma}

{ 
\begin{proof}
The lemma follows directly from the construction of $s(x)$ and $h(x)$; indeed, for any $x \in \mathbb{R}^n$,
$$x^{\top} g(x)=\left(1-x^{\top} \mathbf{1}\right) s(x)=\left(1-x^{\top} \mathbf{1}\right)x^\top H(x) x = x^\top h(x),$$
where the first equality follows from Lemma~\ref{lem:factor}, the second equality follows from Lemma~\ref{lem:scalar_vec}, and the last equality follows from the definition of $h(x)$. 

\hfill{\qed}\end{proof}
}
Finally, we establish the following lemma:


\begin{lemma}
\label{lem:H_Bx}
 Let $h$  be given as in~\eqref{eq:h_def}. 
 Then, there exists a $B\in \sk_{[d]}^{n\times n}$ such that 
\begin{equation}
\label{eq:property}
h(x)=B(x) x, {  \text{ for all }  x\in \mathbf{H}.}
\end{equation}
\end{lemma}
\begin{proof}
   Let $B(x):=H(x)x\mathbf{1}^\top-\mathbf{1}x^\top H(x)$, where $H$ is given in Lemma~\ref{lem:scalar_vec}. 
   Since $H$ is symmetric of degree at most $(d-1)$, $B$ is skew-symmetric of degree at most $d$. 
   It remains to show that $B$ satisfies~\eqref{eq:property}. 
    We have that for any $x\in \mathbf{H}$, 
    \begin{align*}
        B(x)x
        &=(H(x) x \mathbf{1}^{\top}-\mathbf{1} x^{\top} H(x))x\\
        &=H(x) x (\mathbf{1}^{\top}x)-\mathbf{1} (x^{\top} H(x)x)\\
        &=H(x) x-\mathbf{1} (x^{\top} H(x)x)\\
        &=H(x) x-x^{\top} H(x) x \mathbf{1} = h(x),
    \end{align*}
    where the third equality follows from the fact that $\mathbf{1}^\top x = 1$ for $x\in \mathbf{H}$ and the last equality follows from the definition~\eqref{eq:h_def} of $h$. 
    \hfill{\qed}\end{proof}
This completes the proof of Proposition~\ref{prop:reduction}. \hfill{\qed} 

\vspace{.2cm}
\noindent
{\it Further discussions.} 
We note here that Lemma~\ref{lem:H_Bx} can be strengthened as follows: Given an
arbitrary $H\in \mathbf{P}^{n\times n}_{[d-1]}$ (not necessarily symmetric), there exists a $B\in \sk_{[d]}^{n\times n}$ such that 
$$
H(x) x - x^\top H(x) x \mathbf{1} = B(x) x \quad \mbox{for all } x\in \mathbf{H}.
$$
We provide a proof of the fact in the Appendix~\ref{sec:appenB}. 

As a consequence, if $H(x)$ is the payoff matrix (so $H(x)x$ is the payoff vector~\cite{park2019payoff}), then the associated replicator dynamics are given by $$\dot x = \diag(x) (H(x)x - x^\top H(x) x \mathbf{1}),$$ 
which is identical with the following:
$$
\dot x = \diag(x) B(x)x.
$$
In other words, Lemma~\ref{lem:H_Bx} alone (or, more precisely, the strengthened version of it) establishes Theorem~\ref{thm:main} for the special case where $g$ itself can be expressed as 
\begin{equation}
\label{eq:g_form}
g(x) = H(x)x - x^\top H(x) x\mathbf{1}
\end{equation}  

One may wonder whether for every $g$ that satisfies~\eqref{eq:hypothesis}, there always exists an $H\in \mathbf{P}^{n\times n}_{[d-1]}$ such that~\eqref{eq:g_form} holds. The answer is no. We present below a counter example.

\begin{example}
\label{example:2}
Consider the following example with $d=1$, $$g(x)=\left[\begin{array}{c}
-x_1^2+x_1 x_2+x_1 \\
-2x_1^2 \\
-x_1^2
\end{array}\right],$$
where $g\in \mathbf{P}_{[2]}^{n}$,
 and 
 $x^{\top} g(x)=\left(1-x_1-x_2-x_3\right)x_1^2+2 x_1^2 x_2-2 x_1^2 x_2=0$, for any $x\in \mathbf{H}$. 
 We claim that there does not exist an $H\in \mathbf{P}_{[0]}^{n \times n}$ such that $g(x)$ can be written in the form of~\eqref{eq:g_form}.
 Suppose, to the contrary, that such constant matrix $H$ exists; then, it follows that for any $i,j\in \{1,2,3\}$, the difference 
$g_i(x)-g_j(x) = (Hx)_i - (Hx)_j$ 
is a linear function in~$x$, which contradicts the fact that $g_2(x)-g_3(x)=-x_1^2$ is quadratic. We thus conclude that no such constant $H$ exists for~\eqref{eq:g_form} to hold.
 \end{example}



\subsection{Proof of Theorem~\ref{thm:main}}
\label{ssec:sum_proof}
Given $g$, let $g_{(0)}$ be the constant term of $g$. 
We then define 
\begin{equation}
\label{eq:def_bar_g}
    \bar g(x) := (g- g_{(0)}) + g_{(0)}\mathbf{1}^\top x. 
\end{equation}
It is clear that $\bar g$ does not contain a constant term. 
Moreover, since $x^\top \mathbf{1} = 1$ for all $x\in \mathbf{H}$, it holds that $\bar g(x) = g(x)$ for all $x\in \mathbf{H}$. In particular, $\bar g$ satisfies 
\begin{equation}
\label{eq:bar_g_g}
x^\top \bar g(x) = x^\top g(x) = 0, \quad \mbox{for all } x\in \mathbf{H}.
\end{equation}
Next, applying Proposition~\ref{prop:reduction} to $\bar g$, we obtain that there exists an $h\in \mathbf{P}_{[d+1]}^{n}$ and a $B\in \mathbf{\Omega}_{[d]}^{n\times n}$ such that
{ 
\begin{equation}
\label{eq:bar_g_h}
x^\top \bar{g}(x)=x^\top h(x) \mbox{ for all }x\in \mathbb{R}^n,
\mbox{ and  } h(x)=B(x)x, \mbox{ for all } x\in \mathbf{H}.
\end{equation}}
Finally, we let $g'(x):=\bar{g}(x)-h(x)$. Then, by~\eqref{eq:bar_g_g} and~\eqref{eq:bar_g_h}, we have that
$$
x^\top g'(x) = x^\top \bar g(x) - x^\top h(x) = 0, \quad \mbox{for all } x\in \R^n,
$$
so $g'(x)\in \mathbf{Q}_{[d+1]}^{n}$. 
Applying Proposition~\ref{prop:g_skew} to $g'$, we obtain that there exists an $A'(x)\in \sk_{[d]}^{n\times n}$ such that $g'(x)=A'(x)x$. 
We thus conclude that $$g(x)=\bar{g}(x)=(A'(x)+B(x))x=A(x)x,\mbox{ for all }x\in\mathbf{H},$$
with $A := A' + B \in \sk_{[d]}^{n\times n}$ as desired.
\hfill{\qed}

The procedure of constructing the skew-symmetric $A(x)$ such that $g(x)=A(x)x$ is summarized in Algorithm~\ref{alg:main}, { which has the same time complexity as Algorithm~\ref{alg:Aprime} (polynomial in $n$ and exponential in $d$).}

\begin{algorithm}
  \KwIn{ $g(x)\in \mathbf{P}_{[d+1]}^n$}
  Obtain $\bar{g}$ from $g$ via equation~\eqref{eq:def_bar_g}\;
  Compute $x^\top \bar{g}(x)$ and solve~\eqref{eq:g_to_s} for $s(x)$\;
  
  Factorize $s(x)=x^\top H(x)x$ as in Lemma~\ref{lem:scalar_vec}\;
  Compute $h(x)$ as in~\eqref{eq:h_def}\;
  Compute $B(x):=H(x) x \mathbf{1}^{\top}-\mathbf{1} x^{\top} H(x)$\;
  Compute $g'(x)=\bar{g}(x)-h(x)$\;
  Input $g'(x)$ for \textbf{Algorithm~\ref{alg:Aprime}} to obtain $A'(x)$
  
  \KwOut{{ Skew-symmetric payoff matrix $A(x)=A'(x)+B(x)$}}
  \caption{Pseudo-code for deriving $A(x)$ from $g(x)$} 
  \label{alg:main}
\end{algorithm}


    

\section{Conclusion and Outlook}
\label{sec:conclusion}
This paper studies replicator dynamics from an inverse viewpoint, asking whether the underlying payoff landscape can be reconstructed from observed strategy frequencies alone. We have shown that identifiability fails in general, as many distinct payoff matrices can generate the same dynamic behavior. The major contribution of the paper is to show that every replicator dynamics admits a zero-sum population game representation, with the payoff matrix being skew-symmetric. 
Beyond its theoretical interest, the representation offers a practical basis for inferring competitive relationships in settings where only aggregate behavioral data are available. 

{  A technical assumption we take is that the payoff vector $p(x)$ is  polynomial. We believe that our result can be extended to the class of continuous payoff vectors. 
Since the simplex $\Delta^{n-1}$ is compact, by Stone-Weierstrass theorem~\cite{timan2014theory} any continuous function defined on the simplex can be approximated uniformly and arbitrarily well by polynomials. 
However, it remains unknown whether a convergent sequence of polynomial payoff vectors yields a convergent sequence of skew-symmetric, polynomial payoff matrices. We will address this issue in the future work. 
A related issue (and also a future direction) is about the robustness of Algorithm~\ref{alg:main} against perturbation of the input $p(x)$, which matters when $p(x)$ is learned from data and, further, approximated by a polynomial.}  


\begin{appendices}
\section{Nullspace of map $\phi$}
\renewcommand{\thesection}{\Alph{section}}  
\label{sec:appenA}
\begin{lemma}
    Let $\phi$ be given as in~\eqref{eq:def_phi}. Then, the nullspace of $\phi$, when restricted to ${\mathbf{P}_{[0]}^{n\times n}}$, is $\left\{\mathbf{1} v^{\top} \mid v \in \mathbb{R}^n\right\}$.
\end{lemma}
\begin{proof}

We prove the result in two steps, establishing sufficiency and necessity.
\vspace{.1cm}

\noindent
{\it{Proof of sufficiency}.} Let $H=\mathbf{1} v^{\top}$ with some $v \in \mathbb{R}^n$. For any $x\in \mathbb{R}^n$ with $x^\top\mathbf{1}=1$, we have 
$$H x=\mathbf{1}\left(v^{\top} x\right), \quad x^{\top} H x=x^\top \mathbf{1}v^{\top} x= v^{\top} x.$$
We can write
$$\phi(H)=H x-\left(x^{\top} H x\right) \mathbf{1}=\mathbf{1}\left(v^{\top} x\right)-\mathbf{1}\left(v^{\top} x\right)=0.$$
Hence every matrix of the form $\mathbf{1}v^{\top}$ lies in $\ker\phi$.
\vspace{.1cm}

\noindent
{\it{Proof of necessity}.} 
Let $H\in {\mathbf{P}_{[0]}^{n\times n}}$ be in the kernel of $\phi$, i.e., 
$\phi(H)(x)=0$ for all $x \in \mathbb{R}^n$ such that $\mathbf{1}^{\top} x=1$. Consider any $x$ in the relative interior of the simplex (i.e., $x_i > 0$ for all $i$, $\mathbf{1}^\top x=1$). Since $\operatorname{diag}(x)$ is nonsingular, the condition $\phi(H)=0$ implies
\begin{equation}
    H x = \left(x^{\top} H x\right) \mathbf{1}. \label{eq:interior_cond}
\end{equation}
For any $i\neq j$, left-multiplying \eqref{eq:interior_cond} by $(e_i - e_j)^{\top}$ yields
\begin{equation}
    (e_i - e_j)^{\top} H x = \left(x^{\top} H x\right) (e_i - e_j)^{\top} \mathbf{1} = 0, \label{eq:orth_cond}
\end{equation}
where the second equality follows from $(e_i - e_j)^{\top} \mathbf{1} = 0$.
Now, consider any $y \in \mathbb{R}^n$ such that $\mathbf{1}^{\top} y = 0$. Let $x_0$ be a point in the relative interior of the simplex. For a sufficiently small scalar $t$, $x_0 + ty$ remains in the relative interior. By~\eqref{eq:orth_cond}, we have
$$
    (e_i - e_j)^{\top} H (x_0 + ty) = (e_i - e_j)^{\top} H x_0 + t (e_i - e_j)^{\top} H y = 0.
$$
Since $(e_i - e_j)^{\top} H x_0 = 0$, it follows that $(e_i - e_j)^{\top} H y = 0$ and that $(e_i-e_j)^\top H$ lies in the orthogonal complement of $\{y:\mathbf{1}^\top y=0\}$,
which is $\operatorname{span}\{\mathbf{1}\}$. Hence there exists $c$ such that
$$
    (e_i - e_j)^{\top} H = c \mathbf{1}^{\top}.
$$
Right-multiplying by $x_0$ yields $0 = c \mathbf{1}^{\top} x_0 = c$, which implies that $c=0$. Therefore, $e_i^{\top} H = e_j^{\top} H$ for any $i\neq j$, i.e., $H = \mathbf{1} v^{\top}$ for some $v \in \mathbb{R}^n$.
\hfill{\qed} 
\end{proof}

\section{Strengthened Version of Lemma~\ref{lem:H_Bx}}
\label{sec:appenB}
\begin{lemma}
    Given an arbitrary $H \in \mathbf{P}_{[d-1]}^{n \times n}$, there exists a $B \in \boldsymbol{\Omega}_{[d]}^{n \times n}$ such that
$$
H(x) x-x^{\top} H(x) x \mathbf{1}=B(x) x \quad \text { for all } x \in \mathbf{H} .
$$

\end{lemma}
\begin{proof}
We start the proof by writing $H(x)\in \mathbf{P}_{[d-1]}^{n \times n}$ into its symmetric and skew-symmetric components $S(x)\in \mathbf{P}_{[d-1]}^{n\times n}$ and $\Omega(x)\in \sk_{[d-1]}^{n\times n}$ with $H(x)=S(x)+\Omega(x)$. For which, we can obtain by 
$$S(x)=\frac{H(x)+H(x)^\top}{2}
\quad \mbox{and} 
\quad \Omega(x)=\frac{H(x)-H(x)^\top}{2}.$$
With the decomposition, we have
\begin{align*}
    &H(x) x-x^{\top} H(x) x \mathbf{1}
    \\=&(S(x)+\Omega(x)) x-x^{\top} (S(x)+\Omega(x)) x \mathbf{1}\\
    =&S(x)x-x^\top S(x)x\mathbf{1}+\Omega(x)x.
\end{align*}
where the second equation follows from $x^\top\Omega x=0$. By Lemma~\ref{lem:H_Bx}, there exists a $B'(x)\in \sk_{[d]}^{n\times n}$ such that $S(x)x-x^\top S(x)x\mathbf{1}=B'(x)x$.
Define $B(x):=B'(x)+\Omega(x)$, we derive
$$H(x) x-x^{\top} H(x) x \mathbf{1}=B'(x)x+\Omega(x)x=B(x)x.$$
This completes the proof.
\hfill{\qed}\end{proof}

\end{appendices}

\bibliography{feedback}




\end{document}